\documentclass[11pt,twoside]{amsart}
\usepackage{amssymb}
\usepackage{graphicx,color}

\usepackage{tikz}
\usetikzlibrary{matrix}

\newtheorem{theorem}{Theorem}[section]
\newtheorem{proposition}[theorem]{Proposition}

\newtheorem{lemma}[theorem]{Lemma}
\newtheorem{corollary}[theorem]{Corollary}
\newtheorem{definition}[theorem]{Definition}
\newtheorem{notation}[theorem]{Notation}
\newtheorem{example}[theorem]{Example}
\newtheorem{remark}[theorem]{Remark}

\newcommand{\skipit}[1]{{}}
\newcommand{\prfend}{\hbox to7pt{\hfil}
\par\vskip-\baselineskip\hbox to\hsize
{\hfil\vbox {\hrule width6pt height6pt}}\vskip\baselineskip}

\newcommand{\ZZ}{\mathbb{Z}}

\newcommand {\PP}{\mathbb{P}}

\DeclareMathOperator{\Image}{Im}

\newcommand{\af}{\mathbb{A}}

\DeclareMathOperator{\GL}{GL}
\DeclareMathOperator{\codim}{codim}
\DeclareMathOperator{\syz}{syz}
\DeclareMathOperator{\trace}{trace}
\DeclareMathOperator{\GCD}{GCD}
\DeclareMathOperator{\diag}{diag}
\DeclareMathOperator{\HS}{HS}
\DeclareMathOperator{\HF}{HF}
\DeclareMathOperator{\HP}{HP}

\DeclareMathOperator{\Aut}{Aut}
\DeclareMathOperator{\Id}{Id}

\newcommand{\myarrow}[2]{\hbox to #1pt{\hfil$\to$\hfil}{\hskip-#1pt{\raise
10pt\hbox to#1pt{\hfil$\scriptscriptstyle #2$\hfil}}}}

\begin{document}

\title{Togliatti systems associated to the dihedral group and the weak Lefschetz property.}

\author[Liena Colarte]{Liena Colarte-G\'omez}
\address{Departament de Matem\`{a}tiques i Inform\`{a}tica, Universitat de Barcelona, Gran Via de les Corts Catalanes 585, 08007 Barcelona,
Spain}
\email{liena.colarte@ub.edu}

\author[Emilia Mezzetti]{Emilia Mezzetti}
\address{Dipartimento di Matematica e  Geoscienze, Universit\`a di
Trieste, Via Valerio 12/1, 34127 Trieste, Italy}
\email{mezzette@units.it}

\author[Rosa M. Mir\'o-Roig]{Rosa M. Mir\'o-Roig}
\address{Departament de Matem\`{a}tiques i Inform\`{a}tica, Universitat de Barcelona, Gran Via de les Corts Catalanes 585, 08007 Barcelona,
Spain}
\email{miro@ub.edu}

\author[Mart\'{\i} Salat]{Mart\'{\i} Salat-Molt\'o}
\address{Departament de Matem\`{a}tiques i Inform\`{a}tica, Barcelona Graduate School of Mathematics (BGSMath), Universitat de Barcelona, Gran Via de les Corts Catalanes 585, 08007 Barcelona,
Spain}
\email{marti.salat@ub.edu}

\begin{abstract}
In this note, we study Togliatti systems generated by invariants of the dihedral group $D_{2d}$ acting on $k[x_{0},x_{1},x_{2}]$. This leads to the first family of non monomial Togliatti systems, which we call $GT-$systems with group $D_{2d}$. We study their associated varieties $S_{D_{2d}}$, called $GT-$surfaces with group $D_{2d}$. We prove that they are arithmetically Cohen-Macaulay surfaces whose homogeneous ideal, $I(S_{D_{2d}})$, is minimally generated by quadrics and we find a  minimal free resolution of $I(S_{D_{2d}})$.
\end{abstract}

\thanks{Acknowledgments: The first and third authors are partially supported by MTM2016--78623-P. The last author is partially supported by  MDM-2014-0445-18-2. The second author
 is supported by PRIN 2017SSNZAW  and by FRA  of the University of
Trieste, and she is a member of INDAM-GNSAGA}

\maketitle

\tableofcontents

\markboth{}{}

\today

\large

\section{Introduction}

{\em Togliatti systems} were introduced in \cite{MM-RO}, where the authors related the existence of homogeneous artinian ideals failing the weak Lefschetz property to the existence of projective varieties satisfying at least one Laplace equation. Precisely, a Togliatti system is an artinian ideal $I_{d} \subset k[x_{0},\hdots, x_{n}]$ generated by $r \leq \binom{n+d-1}{n-1}$ forms $F_{1},\hdots, F_{r}$ of degree $d$ which fails the weak Lefschetz property in degree $d-1$. The name is in honour of E. Togliatti who gave a complete classification of rational surfaces parameterized by cubics and satisfying at least one Laplace equation of order $2$ (see \cite{T1} and \cite{T2}). Since then, this topic and related problems have been the focus of attention of many works, as one can see in \cite{AMRV}, \cite{Altafi-Boij}, \cite{CMM-R}, \cite{CMM-RS}, \cite{CM-R},
\cite{MM-R1}, \cite{MM-R2}, \cite{MkM-R}, \cite{M-RS} and \cite{Szpond}.
Notwithstanding, most expositions and results deal with {\em monomial} Togliatti systems, while the non monomial case remains barely known.

Recently, in \cite{MM-R2} and \cite{CMM-RS} the authors studied $GT-$systems, a new family of monomial Togliatti systems having a special geometric property. A $GT-$system is a  Togliatti system $I_{d}$ whose associated morphism $\varphi_{I_{d}}:\PP^{n} \to \PP^{\mu_{I_{d}}-1}$ is a Galois covering with cyclic group $\ZZ/d \ZZ$. This geometric property  establishes a new link between Togliatti systems and invariant theory. Precisely in \cite{CMM-R} and \cite{CM-R}, the authors apply invariant theory techniques to investigate both $GT-$systems and their images $X_{d} = \varphi_{I_{d}}(\PP^{n})$, the so called $GT-$varieties. These varieties are actually monomial projections of the Veronese variety $\nu_{d}(\PP^{n}) \subset \PP^{\binom{n+d}{d}-1}$ of $\PP^{n}$ from the linear space $\langle I_{d}^{-1} \rangle$ generated by the Macaulay's inverse system of $I_{d}$. Interest in these varieties relies on the following two problems. The first one is to determine whether $X_{d}$ is an arithmetically Cohen-Macaulay (shortly aCM) monomial projection  of $\nu_{d}(\PP^{n})$. This contributes to the longstanding problem, posed by Gr\"obner in \cite{Grob}, to determine when a projection of $\nu_{d}(\PP^{n})$ is aCM. The second one is the classical problem of finding a minimal free resolution of the homogeneous ideal of a projective variety. In this note, we extend the notion of $GT-$systems and $GT-$varieties, as presented in \cite{CMM-R}, to the action of any finite group, including non abelian ones. This sheds new light on the study of non monomial Togliatti systems. In this work, we focus on the dihedral group action on $k[x_0,x_1,x_2]$. We show that the invariant theory point of view used in \cite{CMM-R}, provides enough techniques to study the $GT-$system $I_{2d}$ associated to the dihedral group and to tackle the geometry of the $GT-$surfaces $S_{D_{2d}}$ defined by the $GT-$system $I_{2d}$.

More precisely, we fix integers $3 \leq d$, $0 < a < \frac{d}{2}$ with $\GCD(d,a) = 1$ and $\varepsilon$ a $2d$th primitive root of $1$. We set $e = \varepsilon^2$ and let $\rho_{a}:D_{2d}\rightarrow \GL(3,k)$ be the linear representation of $D_{2d}=\langle \tau, \eta|\tau^{d}=\eta^{2}=(\eta\tau)^{2}=1\rangle$, the dihedral group of order $2d$, defined by:
$$
\rho_{a}(\tau)=M_{d;a,d-a}=
\left(
\begin{array}{lll}
1&0&0\\
0&e^{a}&0\\
0&0&e^{d-a}
\end{array}
\right)\qquad\text{and}\qquad
\rho_{a}(\eta)=\sigma = \begin{pmatrix}
1 & 0 & 0\\
0 & 0 & 1\\
0 & 1 & 0
\end{pmatrix}.$$
Since $\GCD(a,d) = 1$, the finite cyclic group $\langle \rho_{a}(\tau) \rangle$ of order $d$ coincides with  $\langle M_{d;1,d-1} \rangle \subset \GL(3,k)$. We set $C_{2d}$ the finite cyclic group of order $2d$ generated by $\varepsilon \Id$ and let 
$\overline{D_{2d}} = D_{2d} \times C_{2d} \subset \GL(3,k)$ be the cyclic extension of $D_{2d}$. We denote by $I_{2d}$ the ideal generated by all forms of degree $2d$ which are invariants of $\overline{D_{2d}}$. Our main goal is to relate the ideal $I_{2d}$ to the ring $R^{\overline{D_{2d}}}$ of invariants  of $\overline{D_{2d}}$. Using the structure of the ring $R^{D_{2d}}$, we determine a minimal set of $\mu_{2d}$ generators of $I_{2d}$, formed by monomials and binomials, and we prove that it is a minimal set of generators of $R^{\overline{D_{2d}}}$. This allows us to establish that the ideal $I_{2d}$ is a $GT-$system with group $D_{2d}$, once proved in Lemma \ref{Lemma:TogliattiBound} that $\mu_{2d} \leq 2d+1$  (see Theorem \ref{tea}). As a consequence, we obtain that $R^{\overline{D_{2d}}}$ is the coordinate ring of the surface $S_{D_{2d}} = \varphi_{I_{2d}}(\PP^{2})$ associated to the $GT-$system $I_{2d}$. Through this connection, we prove that $S_{D_{2d}}$ is aCM and we compute a minimal free resolution of $R^{\overline{D_{2d}}}$. In particular, we show that its homogeneous ideal $I(S_{D_{2d}})$ is minimally generated by quadrics and we determine a minimal set of generators.

\vspace{0.3cm}
Let us explain how this note is organized. We begin establishing in Section \ref{Sec:Preliminary}, all the preliminary results and definitions needed in the sequel. In particular, we define the extended notion of $GT-$system with respect to any finite group acting on $k[x_0,\hdots,x_n]$.
Section \ref{the algebra of invariants} is devoted to find a set of fundamental invariants of $\overline{D_{2d}}$. We prove that $R^{\overline{D_{2d}}}$ is minimally generated by monomials and homogeneous binomials of degree $2d$, which we completely determine (see Theorem \ref{Theorem:Algebrabasis}) using the structure of the ring $R^{D_{2d}}$. We compute the Hilbert function, series and polynomial of $R^{\overline{D_{2d}}}$ and we establish that $R^{\overline{D_{2d}}}$ is a level algebra with Castelnuovo-Mumford regularity three. In Section \ref{A new class of non monomial Togliatti systems.}, we introduce a new family of non monomial Togliatti systems, that we call $GT-$systems with group $D_{2d}$ and we study their associated varieties, we call them $GT-$surfaces with group $D_{2d}$. We identify the coordinate ring of any $GT-$surface with group $D_{2d}$ with the ring $R^{\overline{D_{2d}}}$. This allows us to translate geometrically the results obtained in Section \ref{the algebra of invariants}. We show that any $GT-$surface with group $D_{2d}$ is aCM. In addition, the information from the Hilbert series and the regularity allow us to compute a minimal free resolution of the homogeneous ideal $I(S_{D_{2d}})$ of any $GT-$surface $S_{D_{2d}}$ with group $D_{2d}$ (see Theorem \ref{Theorem:Resolution}). In particular, we show that $I(S_{D_{2d}})$ is minimally generated by quadrics. Right after, we focus on determining a minimal set of generators.

\vspace{0.3cm}
\noindent {\bf Notation.}
Through this note $k$ denotes an algebraically closed field of characteristic zero and $\GL(n,k)$ denotes the group of invertible $n \times n$ matrices with coefficients in $k$.

\vspace{0.3cm}
\noindent {\bf Acknowledgments.} The authors are very grateful to the anonymous referee of an earlier version of this work for providing insightful comments and directions which have substantially improved this paper.

\section{Preliminaries.}\label{Sec:Preliminary}
In this section, we collect the main concepts and tools we use in the body of this note. First, we relate the weak Lefschetz property of artinian ideals with varieties satisfying a Laplace equation and we recall the notion of Togliatti system introduced in \cite{MM-RO}. Secondly, we see that quotient varieties by finite groups are Galois coverings and we
extend the notion of $GT-$system from \cite{MM-R2} to any finite group $G$. Finally, we review some basic facts on the theory of invariants of finite groups needed in the sequel.

\vspace{0.3cm}
\noindent {\bf Weak Lefschetz property.} Set $R=k[x_{0},\hdots, x_{n}]$ the polynomial ring and let $I\subset R $ be a homogeneous artinian ideal. We say that $I$ has the \emph{weak Lefschetz property (WLP)}
if there is a linear form $L \in (R/I)_1$ such that, for all
integers $j$, the multiplication map
\[
\times L: (R/I)_{j-1} \to (R/I)_j
\]
has maximal rank, i.e.\ it is injective or surjective.  In \cite{MM-RO} Mezzetti, Mir\'{o}-Roig and Ottaviani proved
that the failure of the WLP is related to the existence of varieties satisfying at least one Laplace
equation of order greater than 2. More precisely, they proved:

\begin{theorem} \label{tea} Let $I\subset R$ be an artinian
ideal
generated
by $r$ forms $F_1,\dotsc,F_{r}$ of degree $d$ and let $I^{-1}$ be its Macaulay inverse system.
If
$r\le \binom{n+d-1}{n-1}$, then
  the following conditions are equivalent:
\begin{itemize}
\item[(i)] $I$ fails the WLP in degree $d-1$;
\item[(ii)]  $F_1,\dotsc,F_{r}$ become
$k$-linearly dependent on a general hyperplane $H$ of $\PP^n$;
\item[(iii)] the $n$-dimensional   variety
 $X=\overline{\Image (
\varphi )}$
where
$\varphi \colon\PP^n \dashrightarrow \PP^{\binom{n+d}{d}-r-1}$ is the  rational map associated to $(I^{-1})_d$,
  satisfies at least one Laplace equation of order
$d-1$.
\end{itemize}
\end{theorem}
\begin{proof} See \cite[Theorem 3.2]{MM-RO}.
\end{proof}

In view of this result, a {\em Togliatti system} is defined as an artinian ideal $I\subset R$ generated by $r\leq \binom{n+d-1}{n-1}$ forms of degree $d$ which fails the WLP in degree $d-1$.
This name is in honor of  Togliatti who proved that the only smooth
Togliatti system of cubics is $$I = (x_0^3,x_1^3,x_2^3,x_0x_1x_2)\subset k[x_0,x_1,x_2]$$
(see \cite{BK}, \cite{T1} and \cite{T2}). The systematic study of Togliatti systems was initiated in \cite{MM-RO}
and for recent results the reader can see \cite{MM-R1},  \cite{MkM-R}, \cite{AMRV}, \cite{M-RS}, \cite{MM-R2}, \cite{CMM-R} and \cite{CMM-RS}. In this paper, we will restrict our attention to a particular case of Togliatti systems, the so called GT-systems which we are going to introduce now.

\vspace{0.3cm}
\noindent {\bf Galois coverings and $GT-$systems.}
Let us recall the notion of a Galois covering. A {\em covering} of a variety $X$ consists of a variety $Y$ and a finite morphism $f: Y \to X$. The {\em  group of deck transformation $G:=\Aut(f)$} is defined to be the group of automorphisms of $Y$ commuting with $f$. We say that $f: Y \to X$ is a covering with group $\Aut(f)$. If the fibre of a covering $f: Y \to X$ over a general point consists of $d$ points we say that $f$ is a covering of degree $d$.
\begin{definition} \rm  A covering $f: Y \to X$ of a variety $X$ is a Galois covering if the group $\Aut(f)$ acts transitively on the fibre $f^{-1}(x)$ for some $x \in X$, and hence for all $x\in X$. We say that $f: Y \to X$ is a Galois covering with group $\Aut(f)$.
\end{definition}

Quotients varieties by finite groups of automorphisms work particularly well with respect to Galois coverings.
\begin{proposition} \label{Proposition: Galois Covering}Let $X$ be a projective variety and $G \subset \Aut(X)$ be a finite group.
If the quotient variety $X/G$ exists, then $\pi: X \to X/G$ is a Galois covering.
\end{proposition}
\begin{proof}
See \cite[Proposition 2.3]{CMM-R}.
\end{proof}

For further details on quotient varieties see for instance \cite{S}.
In \cite{MM-R2} and \cite{CMM-R}, the authors studied a particular class of Togliatti systems arising from actions of the cyclic group over $R$. In particular, let $d\geq 3$ and $1\leq a<b\leq d-1$ be integers such that $\GCD(a,b,d)=1$. We denote by $M_{d;a,b}$ the matrix $\diag(1,e^{a},e^{b})\subset \GL(3,k)$, where $e$ is a $d$th root of $1$. Then, let $\rho_{a,b}:C_{d} \rightarrow \GL(3,k)$ be the representation of $C_{d}=\langle \tau|\tau^{d}=1\rangle$, the cyclic group of order $d$, given by $\rho_{a,b}(\tau)=M_{d;a,b}\subset \GL(3,k)$. With this notation, they proved:

\begin{proposition}
Let $I_{d} \subset k[x_0,x_1,x_2]$ be the ideal generated by all forms of degree $d$ which are invariants of $\rho_{a,b}(C_{d})$. Let $\mu(I_{d})$ denote the minimal number of generators of $I_d$. If $\mu(I_{d})\leq d+1$, then $I_{d}$ is a Togliatti system. Moreover, the associated morphism $\varphi_{I_{d}}:\PP^{2}\rightarrow \PP^{\mu(I_{d})-1}$ is a Galois covering with cyclic group $\ZZ/d \ZZ$.
\end{proposition}
\begin{proof}
See \cite[Corollary 3.5]{CMM-R}.
\end{proof}

The above result motivates the following definition.

\begin{definition}
Let $G$ be a finite group. We say that a Togliatti system $I=(F_{1},\dotsc,F_{r})\subset R$ is a {\em GT-system with group $G$} if the associated morphism $\varphi_{I}:\PP^{n}\rightarrow \PP^{r-1}$ is a Galois covering with group $G$.
\end{definition}

The study of the GT-systems with cyclic group $\ZZ/d\ZZ$ is presented in \cite{CMM-R}, \cite{CMM-RS}, \cite{CM-R} and \cite{MM-R2}. In all these papers, the group $G$ is abelian and the GT-system is monomial. In this note we study GT-systems with non-abelian finite group $G$, more precisely with $G$ the dihedral group, and  we get the first examples of non-monomial GT-system.

\vspace{0.3cm}
\noindent {\bf Invariant theory of finite groups. }
A finite group of automorphisms of the affine space  $\mathbb A^{n+1}$ can be regarded as a finite group
$G \subset \GL(n+1,k)$ acting on the polynomial ring $R$. Let us denote by $R^{G} = \{f \in R \,\mid \, g(f) = f, \; \forall g \in G\}$ the ring of invariants of $G$. The ring $R^G$ inherits the natural grading of $R$, that is $R^{G} = \bigoplus_{t \geq 0} R^{G}_{t}$, where $R_{t}^{G}:= R_{t} \cap R^{G}$. We have the following result. 

\begin{lemma}\label{Corollary: Hilbertfunction GT} Fix $t \geq 1$ and let $G \subset \GL(n+1,k)$ be a finite linear group acting on $R$. Then
$$\dim_{k} R_{t}^{G} = \frac{1}{|G|} \sum_{g \in G} \trace(g^{(t)})$$
where $g^{(t)}$ is the linear map induced by the action of $g$ on $R_{t}$.
\end{lemma}
\begin{proof}
See \cite[Lemma 2.2.2]{Sturmfels}.
\end{proof}

Geometrically, $R^{G}$ is the coordinate ring of the quotient variety of $\mathbb A^{n+1}$ by $G$. To be more precise, let $\{f_{1},\hdots, f_{t}\}$ be a minimal set of  generators of the algebra $R^{G}$, often called a set of {\em fundamental invariants},  and let  $k[w_{1},\hdots, w_{t}]$ be  the polynomial ring in the new variables $w_{1},\hdots, w_{t}$. Then, the quotient variety of $\mathbb A^{n+1}$ by $G$ is given by the morphism $\pi: \mathbb A^{n+1}\to \pi(\mathbb A^{n+1})\subset \mathbb A^t$, such that $\, \pi(a_{0},\hdots, a_{n})= (f_{1}(a_{0},\hdots, a_{n}), \hdots, f_{t}(a_{0},\hdots, a_{n}))$. The ideal $I(\pi(\mathbb A^{n+1}))$ of the quotient variety
is called the {\em ideal of syzygies} among the invariants $f_{1},\hdots,
f_{t}$; it is the kernel of the homomorphism from $R$ to $k[w_{1},\hdots, w_{t}]$ defined by $w_{i} \to f_{i}$, $i = 1,\hdots, t$. We denote it by $\syz(f_{1},\hdots, f_{t})$. We have:

\begin{proposition}\label{Proposition: qv by fg acting linearly on poly ring} Let $G \subset \GL(n+1,k)$ be a finite group acting on $\mathbb A^{n+1}$. Let $f_{1},\hdots, f_{t}$ be a set of fundamental invariants and let $\pi: \af^{n+1} \to \af^{t}$ be the induced morphism. Then,
\begin{itemize}
\item [(i)] $\pi(\af^{n+1})$ is the affine quotient variety by $G$ with affine coordinate ring $R^{G}$.
\item [(ii)] $R^{G} \cong k[w_{1},\hdots, w_{t}]/\syz(f_{1},\hdots, f_{t}).$
\item [(iii)] $\pi$ is a Galois covering of $\pi(\af^{n+1})$ with group $G$. The cardinality of a general orbit $G(a)$, $a \in \af^{n+1}$, is called the degree of the covering.
\end{itemize}
\end{proposition}
\begin{proof} See \cite[Section 6]{Stanley} and Proposition \ref{Proposition: Galois Covering}.
\end{proof}

If we can find a homogeneous set of fundamental invariants $\{f_1,\hdots,f_t\}$ such that $\pi: \PP^n \to \PP^{t-1}$ is a morphism, then the projective version of \ref{Proposition: qv by fg acting linearly on poly ring} is true.

\section{The algebra of invariants of the dihedral group.}
\label{the algebra of invariants}

Through this section, we study the action of the dihedral group on the polynomial ring $R=k[x_0,x_1,x_2]$. We fix integers $3 \leq d$, $0 < a < \frac{d}{2}$ with $\GCD(d,a) = 1$ and $\varepsilon$ a $2d$th primitive root of $1$. We set $e = \varepsilon^2$ and let $\rho_{a}:D_{2d}\rightarrow \GL(3,k)$ be the linear representation of $D_{2d}=\langle \tau, \epsilon|\tau^{d}=\eta^{2}=(\eta\tau)^{2}=1\rangle$, the dihedral group of order $2d$, defined by (see \cite{Serre}):
$$
\rho_{a}(\tau)=M_{d;a,d-a}=
\left(
\begin{array}{lll}
1&0&0\\
0&e^{a}&0\\
0&0&e^{d-a}
\end{array}
\right)\qquad\text{and}\qquad
\rho_{a}(\eta)=\sigma = \begin{pmatrix}
1 & 0 & 0\\
0 & 0 & 1\\
0 & 1 & 0
\end{pmatrix}.$$
It is the direct sum of the trivial representation in $\GL(1,k)$ and a faithful representation in $\GL(2,k)$ of $D_{2d}$. Therefore, the ring $R^{D_{2d}}$ of invariants is generated by the three algebraically independent invariants $x_0$, $x_1x_2$ and $x_1^d + x_2^d$ of $D_{2d}$ (see \cite{Shephard-Todd} and \cite{Stanley}). Thus $R^{D_{2d}} = k[x_0,x_1x_2,x_1^d + x_2^d] = \oplus_{t \geq 0} R_{t} \cap R^{D_{2d}}$ is a non-standard graded polynomial ring. 

Since $\GCD(a,d) = 1$, the finite cyclic group $\langle \rho_{a}(\tau) \rangle$ of order $d$ coincides with  $\Gamma := \langle M_{d;1,d-1} \rangle \subset \GL(3,k)$. We set $C_{2d}$ the finite cyclic group of order $2d$ generated by $\varepsilon \Id$ and we define $\overline{\Gamma} \subset \GL(3,k)$ to be the cyclic extension of $\Gamma$, i.e. $\overline{\Gamma} = \Gamma \times C_{2d}$. Similarly, let 
$\overline{D_{2d}} = D_{2d} \times C_{2d} \subset \GL(3,k)$ be the cyclic extension of $D_{2d}$. We see the ring of invariants $R^{\overline{D_{2d}}}$ of $\overline{D_{2d}}$ as a $k$-graded subalgebra of $R$ and $R^{D_{2d}}$ as follows:
$$R^{\overline{D_{2d}}} = \bigoplus_{t \geq 0} R^{\overline{D_{2d}}}_{t}, \quad \text{where} \quad  R^{\overline{D_{2d}}}_t := R^{D_{2d}}_{2dt} = R_{2dt} \cap R^{D_{2d}}.$$
We relate $R^{\overline{D_{2d}}}$ to the ring $R^{\overline{\Gamma
}}$ studied in \cite{CMM-R} and this connection allows to compute the Hilbert function, Hilbert polynomial and Hilbert series of $R^{\overline{D_{2d}}}$. We also provide a complete description of a homogeneous $k$-basis for each $R^{\overline{D_{2d}}}_{t}$, $t \geq 1$. Our main result shows that the graded $k$-algebra $R^{\overline{D_{2d}}}$ is generated in degree $1$. Let us begin with the following remarks.

\begin{remark} \rm The action of $\sigma$ on a monomial $x_0^{a_0}x_1^{a_1}x_2^{a_2}$ is given by $x_0^{a_0}x_1^{a_2}x_2^{a_1}$. Therefore, the action of $\langle M_{d;1,d-1}^{l}\sigma \rangle \subset \GL(3,k)$ is the same as the action of $\langle \diag(1,e^{d-1},e)^l \rangle \subset \GL(3,k)$ for any $0\leq l\leq d-1$.
\end{remark}

\begin{remark}\label{Remark: equivalent action} \rm
\begin{itemize}
\item [(i)] If a monomial of degree $2dt$ is an invariant of $\Gamma$, it is also an invariant of $\langle \diag(1,e^{d-1},e) \rangle \subset \GL(3,k)$. Actually, a monomial $x_0^{a_{0}}x_1^{a_{1}}x_2^{a_{2}}$ of degree $2dt$ is  an invariant of $\Gamma$ if and only if there exists $r \in \{0,\hdots, 2t(d-1)\}$ such that $(a_{0},a_{1},a_{2}) \in \ZZ_{\geq 0}^{3}$ is a solution of the integer system:
$$\left\{\begin{array}{rcl}
a_{0} + a_{1} + a_{2} & = & 2dt\\
a_{1} + (d-1)a_{2} & = &rd.
\end{array}\right.$$
Now $0 < (d-1)a_{1} + a_{2} = da_{1} - (a_{1} - a_{2}) = d(a_1 + a_{2}) - rd$ is also a multiple of $d$. Hence $(a_{0},a_{1},a_{2})$ is also a solution of the system
$$\left\{\begin{array}{rcl}
a_{0} + a_{1} + a_{2} &  = & 2dt\\
(d-1)a_{1} + a_{2} & = &(a_1 + a_{2}-r)d
\end{array}\right.$$
which implies that $x_0^{a_{0}}x_1^{a_{1}}x_2^{a_{2}}$ is an invariant of $\langle \diag(1,e^{d-1},e) \rangle$.
\item [(ii)] Any monomial of degree $2dt$ of the form $x_0^{2dt-2a_1}x_1^{a_1}x_2^{a_1}$, $a_{1} = 0,\hdots,td$, is an invariant of $\Gamma$. There are exactly $td+1$ monomials of degree $td$ of such form. 
\end{itemize}
\end{remark}

Next, we compute the Hilbert function $\HF(R^{\overline{D_{2d}}},t)$ of $R^{\overline{D_{2d}}}$.  Fix $t \geq 1$. Since $\HF(R^{\overline{D_{2d}}},t)$ is equal to the Hilbert function $\HF(R^{D_{2d}},2dt)$ of $R^{D_{2d}}$,
by Lemma \ref{Corollary: Hilbertfunction GT} we have
$$\HF(R^{\overline{D_{2d}}},t) = \frac{1}{2d} \sum_{g \in \rho(D_{2d})} \trace(g^{(2dt)}) = \frac{1}{2d} \trace(\sum_{g \in \rho(D_{2d})} g^{(2dt)}),$$
where $g^{(2dt)}$ is the restriction of $g$ to $R_{2dt}$.  We choose the set of all monomials of degree $2dt$ as a basis $\mathcal{B}$ of $R_{2dt}$, namely $\mathcal{B}=\{m_{1},\hdots, m_{N}\}$, where $N=\dim_{k}R_{2dt}=\binom{2dt+2}{2}$.

\begin{proposition} \label{Proposition:HilbertFunctionD2d} With the above notation,
 $$\HF(R^{\overline{D_{2d}}},t) = \frac{2dt^2 + (d + \GCD(d,2) + 2)t + 2}{2}.$$
\end{proposition}
\begin{proof}
Let $m_{i}$ be a monomial in $\mathcal{B}$ of degree $2dt$. We denote by $M$ the matrix which represents the linear map $\sum_{g \in \rho(D_{2d})}g^{(2dt)}$ in the above basis. We distinguish two cases.

\vspace{0.5cm}
\noindent\underline{Case 1:} $m_{i} \in R^{\Gamma}$. Then by Remark
\ref{Remark: equivalent action},
$$M_{(i,i)} = \left\{\begin{array}{rcl}
2d & \text{if}& \sigma(m_{i}) = m_{i}\\
d           & \text{if}& \sigma(m_{i}) \neq m_{i}\\
\end{array}\right.$$

\vspace{0.5cm}
\noindent\underline{Case 2:} $m_{i} \notin R^{\Gamma}$. If $\xi$ is a $d$th root of unity, we have the equality
$1 + \xi + \cdots + \xi^{d-1} = 0$. This, in addition to Remark
\ref{Remark: equivalent action}, gives $M_{(i,i)} = 0$.

Let $\mu_{2dt}^{c}$ be the number of monomials of degree $2dt$ in $R^{\Gamma}$. Thus, we have obtained that
$$(2d)\HF(R^{\overline{D_{2d}}},t) = d(\mu_{2dt}^{c} + td+1)).$$
 By \cite[Theorem 4.5]{CMM-R}, $\mu_{2dt}^{c}$ is equal to the Hilbert function $\HF(R^{\overline{\Gamma}},2t)$ of the ring $R^{\overline{\Gamma}}$ in degree $2t$. By \cite[Theorem 4.11]{CMM-R}
$\mu_{2dt}^{c} = 2dt^2 + 2t+\GCD(2,d)t+1,$ which completes the proof.
\end{proof}

From the above result, we directly obtain the Hilbert polynomial $\HP(R^{\overline{D_{2d}}},t)$ and the Hilbert series $\HS(R^{\overline{D_{2d}}},z)$ of the ring $R^{\overline{D_{2d}}}$.
\begin{proposition} \label{Proposition:Hilbert Function Series} With the above notation, the Hilbert polynomial and the Hilbert series of $R^{\overline{D_{2d}}}$ are given by the following expressions:
\begin{itemize}
\item[(i)] $\begin{array}{ll}
 \hspace{0.5cm} & \displaystyle{\HP(R^{\overline{D_{2d}}},t) = \frac{1}{2}(2dt^2 + (d + \GCD(d,2)+2)t + 2)} 
 \end{array}$
 
 \vspace{0.2cm}
 \item[(ii)] $\begin{array}{ll}
 \hspace{0.5cm} & \displaystyle{\HS(R^{\overline{D_{2d}}},z) =\frac{1}{(1-z)^3} \left(\frac{d-\GCD(d,2)}{2}z^2 + \frac{3d+\GCD(d,2)-2}{2}z + 1 \right).}
\end{array}$
\end{itemize}
 
\end{proposition}

Next, we use the relation between the two rings $R^{\overline{D_{2d}}}$ and $R^{\overline{\Gamma}}$ to determine a $k$-basis of any vector space $R^{\overline{D_{2d}}}_{t}$, $t \geq 1$, and to find a $k$-algebra basis of $R^{\overline{D_{2d}}}$.
Fix $t \geq 1$ and consider $R^{\overline{D_{2d}}}_{t}.$

\begin{proposition} \label{Proposition:GradedBasisInvariant} A $k$-basis $\mathcal{B}_{2dt}$ of the vector space $R^{\overline{D_{2d}}}_{t}$ is formed by:
\begin{itemize}
\item[(i)] the set of $td + 1$ monomial invariants $x_0^{2dt},x_0^{2dt-2}x_1x_2,
x_0^{2dt-4}x_1^2x_2^2, \hdots, x_1^{td}x_2^{td}$ of $\Gamma$ of degree $2dt$; and
\item[(ii)] the set of all binomials  $x_0^{a_{0}}x_1^{a_{1}}x_2^{a_{2}}+ x_0^{a_{0}}x_1^{a_{2}}x_2^{a_{1}}$ of degree $2dt$ such that $a_{1} \neq a_{2}$ and $x_0^{a_{0}}x_1^{a_{1}}x_2^{a_{2}} \in R^{\Gamma}$.
\end{itemize}
\end{proposition}

\begin{proof} By Remark \ref{Remark: equivalent action}(i), there are exactly 
${\frac{1}{2}}(\mu_{2dt}^{c} - td-1)$ binomials  of the form $x_0^{a_{0}}x_1^{a_{1}}x_2^{a_{2}} + x_0^{a_{0}}x_1^{a_{2}}x_2^{a_{1}} \in R^{\Gamma}$ of degree $2dt$ with $a_{1} \neq a_{2}$. Since the forms in $(i)$ and $(ii)$ are $k$-linearly independent, the result follows from Proposition \ref{Proposition:HilbertFunctionD2d}.
\end{proof}

We illustrate the above result with a couple of examples.

\begin{example} \label{Example: Basist1t2}\rm

\begin{itemize}
\item[(i)] For $d = 3$ and $D_{2\cdot3} = \langle M_{3;1,2},\sigma \rangle$, we have:
\smallskip

\noindent \hspace{-0.5cm} $\mathcal{B}_{2\cdot 3} = \{x_0^{6}, x_0^{3}x_1^{3} + x_0^3x_2^{3}
, x_0^{4}x_1x_2, x_1^{6}+x_2^{6}, x_0x_1^4x_2+x_0x_1x_2^4, x_0^2x_1^2x_2^2, x_1^3x_2^3\}$, $\HF(R^{\overline{D_{2\cdot 3}}},1) = 7$.
\smallskip

\noindent \hspace{-0.5cm} $\mathcal{B}_{4\cdot 3} = \{x_0^{12},x_0^9x_1^3+x_0^{9}x_2^{3} ,x_0^{10} x_1x_2, x_0^6 x_1^6 + x_0^{6}x_2^{6}, $ $x_0^7 x_1^4 x_2 + x_0^{7} x_1x_2^{4} ,x_0^8 x_1^2 x_2^2,x_0^3 x_1^9 + x_0^3 x_2^9, x_0^{4}x_1^7x_2 + x_0^4x_1x_2^7, x_0^5 x_1^5 x_2^2 + x_0^{5}x_1^2x_2^5 ,x_0^6 x_1^3 x_2^3,x_1^{12}+x_2^{12},x_0 x_1^{10} x_2 + x_0x_1x_2^{10} ,x_0^2 x_1^8 x_2^2 + x_0^2 x_1^2 x_2^8, x_0^3 x_1^6 x_2^3 + x_0^3 x_1^3 x_2^6, $ $x_0^4 x_1^4 x_2^4, $ $x_1^9 x_2^3 + x_1^3x_2^9 ,x_0 x_1^7 x_2^4 + x_0 x_1^4x_2^7 ,x_0^2 x_1^5 x_2^5,x_1^6 x_2^6\}$, $\HF(R^{\overline{D_{2\cdot 3}}},2) = 19$.
\smallskip

\item[(ii)] For $d = 4$ and $D_{2 \cdot 4} = \langle M_{4;1,3},\sigma \rangle$, we have:
\smallskip

\noindent \hspace{-0.5cm} $\mathcal{B}_{2\cdot 4} = \{x_0^{8}, x_0^{4}x_1^{4} +
x_0^{4}x_1^{4}, x_0^{6}x_1x_2, x_1^{8}+x_2^{8}, x_0^{2}x_1^{5}x_2+x_0^{2}x_1x_2^{5}, x_0^{4}x_1^{2}x_2^{2}, x_1^{6}x_2^{2} + x_1^{2}x_2^{6}, x_0^{2}x_1^{3}x_2^{3},$ $ x_1^{4}x_2^{4}\}$, $\HF(R^{\overline{D_{2 \cdot 4}}},1) = 9$.
\smallskip

\noindent \hspace{-0.5cm} $\mathcal{B}_{4\cdot 4} =\{x_0^{16},x_0^{12} x_1^4 + x_0^{12} x_2^{4} ,x_0^{14} x_1 x_2, x_0^8 x_1^8 + x_0^8x_2^8, x_0^{10} x_1^5 x_2 + x_0^{10}x_1x_2^5 ,x_0^{12} x_1^2 x_2^2,x_0^4 x_1^{12} + x_0^4 x_2^{12} ,x_0^6 x_1^9 x_2 + x_0^6x_1x_2^9,x_0^8 x_1^6 x_2^2 + x_0^{8}x_1^2x_2^6,x_0^{10} x_1^3 x_2^3,x_1^{16}+x_2^{16}, x_0^2 x_1^{13} x_2 + x_0^2 x_1 x_2^{13}, x_0^4 x_1^{10} x_2^2 + x_0^4x_1^2 x_2^{10} ,x_0^6 x_1^7 x_2^3 + x_0^6x_1^3x_2^7, x_0^8 x_1^4 x_2^4,x_1^{14}x_2^2 + x_1^2x_2^{14} ,x_0^2 x_1^{11} x_2^3 + x_0^2 x_1^3 x_2^{11} ,x_0^4 x_1^8 x_2^4 + x_0^4 x_1^4 x_2^8,x_0^6 x_1^5 x_2^5, $ $x_1^{12} x_2^4 + x_1^4x_2^{12},x_0^2 x_1^9 x_2^5 + x_0^2x_1^5x_2^9, x_0^4 x_1^6 x_2^6,x_1^{10}x_2^6 + x_1^6x_2^{10}, x_0^2 x_1^7 x_2^7,x_1^8 x_2^8\}$, $\HF(R^{\overline{D_{2\cdot 4}}},2) = 25$.
\end{itemize}
\end{example}

The next goal is to prove that $\mathcal{B}_{2d}$ is a set of fundamental invariants of $\overline{D_{2d}}$. To achieve it, we use the natural structure of $R^{\overline{D_{2d}}}$ as a subring of $R^{D_{2d}}$. We set $y_0 = x_0, \, y_1 = x_{1}x_{2}$ and $y_2 = x_1^d + x_2^d$, as we have seen $R^{D_{2d}} = k[y_0,y_1,y_2]$ with $\deg(y_0) = 1$, $\deg(y_1) = 2$ and $\deg(y_2) = d$. With this notation, for any $t \geq 1$ we have that $R^{\overline{D_{2d}}}_t = k[y_0,y_1,y_2]_{2td}$ is the $k$-vector space with monomial basis $\mathcal{A}_{2dt} = \{y_0^{b_0} y_1^{b_1} y_2^{b_2} \,\mid \, b_0 + 2b_1 + db_2 = 2td\}$. In particular, for $t = 1$ we have the change of basis $\rho: k[y_0,y_1,y_2]_{2d} \to R^{\overline{D_{2d}}}_1$
\begin{equation}\label{Equation: change of variable}
\left\{\begin{array}{lllll}
y_0^{b_0}y_1^{b_1}y_2^{b_2} & \mapsto & x_0^{b_0}x_1^{b_1}x_2^{b_1}(x_1^{d} + x_2^{d})^{b_2}, & \text{if} & 0 \leq b_2 \leq 1\\
y_2^{2} & \mapsto & (x_1^{2d} + x_2^{2d}) + 2x_1^{d}x_2^{d}. & &  
\end{array}\right.
\end{equation}

\begin{theorem}\label{Theorem:Algebrabasis} $\mathcal{B}_{2d}$ is a set of fundamental invariants of $\overline{D_{2d}}$. 
\end{theorem}
\begin{proof} We see that for any $t \geq 2$, any monomial $y_0^{b_0}y_1^{b_1}y_2^{b_2} \in \mathcal{A}_{2dt}$ is divisible by a monomial of $\mathcal{A}_{2d}$. Then by induction, it follows that $\mathcal{A}_{2d}$ is a set of generators of  $R^{\overline{D_{2d}}} \subset R^{D_{2d}}$. Using (\ref{Equation: change of variable}), we obtain that $\mathcal{B}_{2d}$ is a minimal set of generators of $R^{\overline{D_{2d}}}$. 

 Let $m = y_0^{b_0}y_1^{b_1}y_2^{b_2} \in \mathcal{A}_{2dt}$ be a monomial of degree $b_0 + 2b_1 + db_2 = 2dt, \, t \geq 2$. On one hand, we may suppose that $b_0 < 2d$, $b_1 < d$ and $b_2 < 2$. Otherwise, $y_0^{2d}$, $y_1^{d}$ or $y_{2}^2$ divide $m$ and the result follows. On the other hand, if $b_2 = 0$,\,  $b_0 < 2d$ and $b_1 < d$, then we have $\deg(m) = 2d$ and $t = 1$. Therefore it only remains to prove the case $b_0 < 2d$, $b_1 < d$ and $b_2 = 1$ with $b_0 + 2b_1 + d = 4d$. Since $b_0 + 2b_1 = 3d$ and $b_1 < d$, this implies that $b_0 \geq d$ and then $y_0^dy_2 \in \mathcal{A}_{2d}$ divides $m$, as required. 
\end{proof}

As a consequence, $R^{\overline{D_{2d}}} \subset R$ is a graded $k$-algebra generated in degree $1$ and $\rho$ induces an isomorphism of graded $k$-algebras $\rho: k[\mathcal{A}_{2d}] \to k[\mathcal{B}_{2d}]$. The following example illustrates Theorem \ref{Theorem:Algebrabasis}. 

\begin{example}\rm We express the invariants of $\mathcal{B}_{4\cdot 3}$ in terms of $\mathcal{B}_{2d}$ (see Example \ref{Example: Basist1t2}(i)). We have 
$$\begin{array}{lll}
\mathcal{A}_{2\cdot 3} & = & \{y_0^{6}, y_0^3y_2, y_0^4y_1, y_2^2, y_0y_1y_2, y_0^2y_1^2, y_1^3\}\\
\mathcal{A}_{4 \cdot 3} & = & \{y_0^{12}, y_0^{10}y_1, y_0^{8}y_1^2, y_0^6y_1^3, y_0^4 y_1^4, y_0^2y_1^5,y_1^6, y_0^9y_2, y_0^7y_1y_2, y_0^5y_1^2y_2, y_0^3 y_1^3 y_2, y_0y_1^4y_2, y_0^6y_2^2, y_0^4y_1y_2^2,\\
& &  y_0^2y_1^2y_2^2, y_1^3y_2^2, y_0^3y_2^3, y_0y_1y_2^3, y_2^4\}.
\end{array}$$
Expressing all monomials of $\mathcal{A}_{4\cdot 3}$ as products of monomials of $\mathcal{A}_{2 \cdot 3}$ and applying (\ref{Equation: change of variable}) we obtain the following factorizations: 

$$\begin{array}{lll}
x_0^{12} & = & (x_0^6)(x_0^6)\\
x_0^{10}x_1x_2 & = &  (x_0^6)(x_0^{4}x_1x_2)\\
x_0^{8}x_1^{2}x_1^{2} & = & (x_0^{6})(x_0^{2}x_1^{2}x_2^{2})\\
x_0^{6}x_1^{3}x_2^{3} & = & (x_0^{6})(x_1^3x_2^3)\\
x_0^{4}x_1^{4}x_2^{4} & = & (x_0^4x_1x_2)(x_1^3x_2^3)\\
x_0^{2}x_1^{5}x_2^{5} & = & (x_0^{2}x_1^{2}x_2^{2})(x_1^{3}x_2^{3})\\
x_1^{6}x_2^{6}      & = & (x_1^3x_2^3)^{2}\\
x_0^9x_1^3+x_0^{9}x_2^{3} & = & x_0^6(x_0^3x_1^3+x_0^3x_2^3)\\
x_0^6 x_1^6 + x_0^{6}x_2^{6}& = &x_0^{6}(x_1^6+x_1^6)\\
x_0^7 x_1^4 x_2 + x_0^{7} x_1x_2^{4}& = & x_0^{4}x_1x_2(x_0^{3}x_1^{3} + x_0^{3}x_2^{3})\\
x_0^{4}x_1^7x_2 + x_0^4x_1x_2^7 & = & x_0^{4}x_1x_2(x_1^{6} + x_2^{6}) \\
x_0^5 x_1^5 x_2^2 + x_0^{5}x_1^2x_2^5 & = & x_0^2x_1^2x_2^2(x_0^3x_1^3+x_0^3x_2^3)\\
x_0^2 x_1^8 x_2^2 + x_0^2 x_1^2 x_2^8& = & x_0^2x_1^2x_2^2(x_1^6+x_2^6)\\
x_0^3 x_1^6 x_2^3 + x_0^3 x_1^3 x_2^6 & = & x_1^3x_2^3(x_0^3x_1^3+x_0^3x_2^3)\\
x_1^9 x_2^3 + x_1^3x_2^9& = &x_1^3x_2^3(x_1^6+x_2^6)\\
x_0 x_1^7 x_2^4 + x_0 x_1^4x_2^7& = &x_1^3x_2^3(x_0x_1^4x_2 + x_0x_1x_2^4)\\
x_0^3 x_1^9 + x_0^3 x_2^9 & = &(x_1^{6}+x_2^{6})(x_0^{3}x_1^{3} + x_0^{3}x_2^{6}) - x_1^{3}x_2^{3}(x_0^{3}x_1^{3} + x_0^{3}x_2^{3})\\
x_1^{12} + x_2^{12} & = & (x_1^{6} + x_2^{6})^2 - 2(x_1^{3}x_2^{3})^{2}\\
x_0 x_1^{10} x_2 + x_0x_1x_2^{10}& = & (x_0x_1^4x_2+ x_0x_1x_2^{4})(x_1^{6} + x_2^{6}) - x_1^{3}x_2^{3}(x_0x_1^4x_2 + x_0x_1x_2^{4}).\\
\end{array}$$
Notice that these decompositions are not unique, for instance $x_0^{8}x_1^{2}x_2^{2}$ can also be factored as $(x_0^{4}x_1x_2)^2$.
\end{example}

We end this section with a corollary regarding the Cohen-Macaulayness of $R^{\overline{D_{2d}}}$.
\begin{corollary}\label{Corollary:CMness}
\begin{itemize}
\item[(i)] $R^{\overline{D_{2d}}}$ is a Cohen-Macaulay level algebra with Cohen-Macaulay type $\frac{1}{2}(d-\GCD(d,2))$ and Castelnuovo-Mumford regularity $3$.

\vspace{0.1cm}
\item[(ii)] $R^{\overline{D_{2d}}}$ is Gorenstein if and only if $d=3$ or $4$.
\end{itemize}
\end{corollary}
\begin{proof}
Since $R^{\overline{D_{2d}}}$ is the ring of invariants by the action of the linear finite group $\overline{D_{2d}}$ on $R$, it is Cohen-Macaulay (see \cite[Proposition 12]{Hochster-Eagon}). The other results follow from that and Proposition \ref{Proposition:Hilbert Function Series}.
\end{proof}

\section{Togliatti systems associated to the dihedral group.}
\label{A new class of non monomial Togliatti systems.}

In this section, we describe Togliatti systems $I_{2d}$ associated to $D_{2d}$ and we study the geometry of their associated varieties. As far as we know, the  $GT-$systems described in previous works were all monomial and
it is worthwhile to point out that $I_{2d}$ is the first large class of non monomial $GT-$systems. Namely, we prove that the ideal generated by the set of fundamental invariants $\mathcal{B}_{2d}$ (see Proposition \ref{Proposition:GradedBasisInvariant} and Theorem \ref{Theorem:Algebrabasis}) of $\overline{D}_{2d}$ is a $GT-$system with group $D_{2d}$. We connect the ring $R^{\overline{D_{2d}}}$ to the coordinate ring of the associated varieties of these $GT-$systems.

Let $I_{2d} \subset R$ be the homogeneous ideal generated by $\mathcal{B}_{2d}$ and we set $\mu_{2d} := \HF(R^{\overline{D_{2d}}},1)$. We denote by $\varphi_{I_{2d}}: \PP^{2} \to \PP^{\mu_{2d}-1}$ the morphism induced by $I_{2d}$. Our first goal is to show that $I_{2d}$ is a non monomial $GT-$system with group $D_{2d}$. We set $S_{D_{2d}} := \varphi_{I_{2d}}(\PP^{2})$ and we call it a {\em $GT-$surface with group $D_{2d}$}. The ring of invariants $R^{\overline{D_{2d}}}$  is then  the coordinate ring of $S_{D_{2d}}$ (see Proposition \ref{Proposition: qv by fg acting linearly on poly ring} and Theorem \ref{Theorem:Algebrabasis}). We study the homogeneous ideal $I(S_{D_{2d}})$ of $S_{D_{2d}}$, which is the prime ideal $\syz(\mathcal{B}_{2d})$ of syzygies among $\mathcal{B}_{2d}$. We compute a minimal free resolution of $I(S_{D_{2d}})$ and, as a consequence, we prove that $I(S_{D_{2d}})$ is minimally generated by quadrics.

\begin{lemma}\label{Lemma:TogliattiBound} For all $d \geq 3$, $\mu_{2d} \leq 2d+1$.
\end{lemma}
\begin{proof}
From Proposition \ref{Proposition:HilbertFunctionD2d},
$$\mu_{2d} = \frac{2d + d + 2 + \GCD(d,2) + 2}{2} \leq \frac{3d + 6}{2},$$ which is smaller or equal than $2d+1$ for all $d \geq 3$.
\end{proof}

\begin{proposition}
The ideal $I_{2d}$ is a $GT-$system with group $D_{2d}$.
\end{proposition}
\begin{proof}
Firstly, notice that $I_{2d}$ is an artinian non-monomial ideal.
Precisely, $x_0^{2d}, x_1^{2d}+x_2^{2d}$ and $x_1^{d}x_2^{d}$ form a homogeneous system of parameters of $R^{\overline{D_{2d}}}$. And secondly, since $\mathcal{B}_{2d}$ is a set of fundamental invariants of $\overline{D_{2d}}$ (see Theorem \ref{Theorem:Algebrabasis}), by Proposition \ref{Proposition: qv by fg acting linearly on poly ring} it follows that $\varphi_{I_{2d}}$ is a Galois covering with group $D_{2d}$. Now, by Lemma \ref{Lemma:TogliattiBound}, $I_{2d}$ is generated by $\mu_{2d} \leq 2d+1$ homogeneous forms of degree $2d$. Hence, by Theorem \ref{tea} it only remains to see that it fails the weak Lefschetz property from degree $2d-1$ to $2d$. Let $L\in (R/I_{2d})_{1}$ be a linear form, and let us consider $F_{2d-1}:=\prod_{\substack{g\in D_{2d}\\ g\neq \Id}}g(L)$. By construction, for any element $g \in D_{2d}$ we have that $g(LF_{2d-1}) = LF_{2d-1}$. Thus $LF_{2d-1}$ is an invariant of $\overline{D_{2d}}$ and the result follows from Theorem \ref{Theorem:Algebrabasis}. 
\end{proof}

In the rest of this section,  we deal with the geometry of $S_{D_{2d}}$. Let us begin with some properties which follow directly from the results we obtained in Section \ref{the algebra of invariants}.

\begin{proposition}
$S_{D_{2d}}$ is an arithmetically Cohen-Macaulay surface of degree $\deg(S_{D_{2d}}) = 2d$, regularity $3$, codimension
${\frac{1}{2}}(3d+\GCD(d,2)-2)$ and Cohen-Macaulay type ${\frac{1}{2}}(d-\GCD(d,2))$.
In particular, $S_{D_{2d}}$ is Gorenstein if and only if $d = 3,4$.
\end{proposition}
\begin{proof}
See Corollary \ref{Corollary:CMness}.
\end{proof}

Our next goal is to determine a minimal free resolution of $I(S_{D_{2d}})$. In particular, we obtain that $I(S_{D_{2d}})$ is generated by quadrics. Let us begin introducing some new notation.

\vspace{0.2cm}
We set $\mathcal{W}_{d}:= \{w_{(r,\gamma)} \;\mid \; 0 \leq r \leq 2(d-1) \;\; \text{and} \;\; max\{0,
\lceil\frac{(r-2)d}{d-2} \rceil\} \leq \gamma \leq r\}$, a set of variables ordered lexicographically. As we will explicitly see in Notation \ref{new variables}, each pair $(r,\gamma)$ as in $\mathcal{W}_{d}$ uniquely determines the exponents of an element in $\mathcal{B}_{2d}$ (see Remark \ref{Remark: equivalent action}  and Proposition \ref{Proposition:GradedBasisInvariant}).
Hence, the cardinality of $\mathcal{W}_{d}$ is
$\mu_{2d}=d+2+\frac{d+\GCD(2,d)}{2}$. We exhibit a few examples.

\begin{example} \rm
\begin{itemize}
\item[(i)] For $d = 3$, $\mathcal{B}_{2\cdot 3} = \{x_0^{6}, x_0^{3}x_1^{3} + x_0^3x_2^{3}, x_0^{4}x_1x_2, x_1^{6}+x_2^{6}, x_0x_1^4x_2+x_0x_1x_2^4, x_0^2x_1^2x_2^2, $ $ x_1^3x_2^3\}$ and
$$\mathcal{W}_{3} = \{w_{(0,0)}, w_{(1,0)},w_{(1,1)},w_{(2,0)},w_{(2,1)},
w_{(2,2)},w_{(3,3)}\}$$
\item[(ii)] For $d = 4$, $\mathcal{B}_{2\cdot 4} = \{x_0^{8}, x_0^{4}x_1^{4} +
x_0^{4}x_1^{4}, x_0^{6}x_1x_2, x_1^{8}+x_2^{8}, x_0^{2}x_1^{5}x_2+x_0^{2}x_1x_2^{5}, x_0^{4}x_1^{2}x_2^{2}, x_1^{6}x_2^{2} + x_1^{2}x_2^{6}, x_0^{2}x_1^{3}x_2^{3},$ $ x_1^{4}x_2^{4}\}$ and
$$\mathcal{W}_{4} = \{w_{(0,0)},w_{(1,0)},w_{(1,1)},w_{(2,0)},w_{(2,1)},w_{(2,2)},w_{(3,2)}, w_{(3,3)},w_{(4,4)}\}.$$
\item[(iii)] For $d = 5$, $\mathcal{B}_{2\cdot 5} = \{x_0^{10}, x_0^5x_1^5+x_0^5x_2^5,
x_0^8x_1x_2, x_1^{10}+x_2^{10}, x_0^3x_1^6x_2 + x_0^3x_1x_2^6, x_0^6x_1^2x_2^2, x_0x_1^7x_2^2+x_0x_1^2x_2^7, x_0^4x_1^3x_2^3, x_0^2x_1^4x_2^4, x_1^5x_2^5\}$ and
$$\mathcal{W}_{5} = \{w_{(0,0)}, w_{(1,0)}, w_{(1,1)}, w_{(2,0)}, w_{(2,1)}, w_{(2,2)}, w_{(3,2)}, w_{(3,3)}, w_{(4,4)}, w_{(5,5)}\}.$$
\end{itemize}
\end{example}

\begin{notation}\rm \label{new variables}
We denote by $S = k[w_{(r,\gamma)}]_{w_{(r,\gamma)} \in \mathcal{W}_{d}}$ the polynomial ring.
The homogeneous ideal $I(S_{D_{2d}})$ is the kernel of the ring homomorphism $\varphi_{d}: S \to k[\mathcal{B}_{2d}]$ defined as follows:
$$\varphi_{d}(w_{(r,\gamma)}) =  \left\{\begin{array}{ll}
x_0^{2d-2\gamma}x_1^{\gamma}x_2^{\gamma}=:m_{(r,\gamma)} & \quad \text{if}\; r = \gamma\\
x_0^{(2-r)d + (d-2)\gamma}(x_1^{rd-(d-1)\gamma}x_2^{\gamma} + x_1^{\gamma}x_2^{rd-(d-1)\gamma})=: m_{(r,\gamma)} + \overline{m_{(r,\gamma)}} & \quad \text{otherwise}.
\end{array}\right.$$
\end{notation}

The information from the Hilbert function of $S_{D_{2d}}$ and the regularity allows us to determine a minimal free resolution of $S_{D_{2d}}$.
Set
$$C := \codim(S_{D_{2d}}) = \frac{3d+\GCD(d,2)-2}{2},\;\text{and}$$
$$h := \deg(S_{D_{2d}}) - C - 2 = \frac{d-\GCD(d,2)-2}{2}.$$

\begin{theorem}\label{Theorem:Resolution} With the above notation, $I(S_{D_{2d}})$ has a minimal free $S$-resolution
$$0 \to S^{b_{C,2}}(-C-2) \to \oplus_{l=1,2} S^{b_{C-1},l}(-C+1-l) \to
\oplus_{l=1,2} S^{b_{C-2,l}}(-C+2-l)$$
$$\to \cdots \to \oplus_{l=1,2} S^{b_{C-h},l}(-C+h-l) \to
S^{b_{C-h-1,1}}(-C+h)$$
$$\to \cdots \to S^{b_{1,1}}(-2) \to S \to S/I(S_{2d}) \to 0$$
where
$$b_{i,j-i}:= \left\{ \begin{array}{lll}
i\binom{C}{i+1} + (C-i-h)\binom{C}{i-1}& &\text{if}\;\; 1 \leq i \leq C-h-1, j = i+1\\
i\binom{r}{i+1} & &\text{if}\;\; C - h \leq i \leq C, j = i+1\\
(i - C + h + 1)\binom{C}{i} & &\text{if} \;\; C-h \leq i \leq C, j = i+2\\
0 && \text{otherwise}
\end{array}\right.$$
\end{theorem}
\begin{proof} For $d = 3,4$ we explicitly compute the resolutions of $S_{D_{2\cdot 3}}$ and $S_{D_{2\cdot 4}}$ in Example \ref{Example:Resolutions}(i),(ii). For all $d \geq 5$ we check that $C+3 \leq 2d \leq 2C$ and then we apply \cite[Corollary 3.4(ii)]{Yanagawa}. Clearly $2d \leq 3d + \GCD(d,2)-2$ for all $d \geq 3$. On the other hand, $$C+3 = \frac{3d + \GCD(d,2)+4}{2} \leq 2d$$ if and only if $3d + \GCD(d,2) + 4 \leq 4d$ if and only if $\GCD(d,2) + 4 \leq d$. The last inequality holds for all $d \geq 5$.
\end{proof}

\begin{corollary}\label{Corollary: quadrics} $I(S_{D_{2d}})$ is minimally generated by
$\frac{9d^2+2d+8}{8}$ quadrics if $d$ is even and by
$\frac{9d^2-4d+3}{8}$ quadrics if $d$ is odd.
\end{corollary}

Let us illustrate Theorem \ref{Theorem:Resolution} with some examples, which we compute using the software Macaulay2 (\cite{M}).

\begin{example} \label{Example:Resolutions}\rm \begin{itemize}
\item[(i)] For $d = 3$, \; $S_{D_{2d}}$ has codimension $C = 4$ and degree $\deg(S_{D_{2d}}) = 6$, so $h = 0$.  A minimal free resolution of $I(S_{D_{2\cdot 3}})$ over $S = k[w_{(r,\gamma)}]_{w_{(r,\gamma)} \in \mathcal{W}_{3}}$ is
$$0 \to S(-6) \to S^{9}(-4) \to S^{16}(-3) \to S(-2)^{9} \to S \to S/I(S_{D_{2\cdot 3}}) \to 0.$$

\item[(ii)] For $d = 4$, \; $S_{D_{2d}}$ has codimension $C = 6$ and degree $\deg(S_{D_{2d}}) = 8$, so $h = 0$. A minimal free resolution of $I(S_{D_{2\cdot 4}})$ over $S = k[w_{(r,\gamma)}]_{w_{(r,\gamma)} \in \mathcal{W}_{4}}$ is
$$0 \to S(-8) \to S^{20}(-6) \to S^{64}(-5) \to S^{90}(-4) \to $$
$$S^{64}(-3) \to S^{20}(-2) \to S  \to S/I(S_{D_{2\cdot 4}}) \to 0.$$

\item[(iii)] For $d = 5$, \; $S_{D_{2d}}$ has codimension  $C = 7$ and degree $\deg(S_{D_{2d}}) = 10$, so we have $h = 1$ and a minimal free resolution of $I(S_{D_{2\cdot 5}})$ over $S = k[w_{(r,\gamma)}]_{w_{(r,\gamma)} \in \mathcal{W}_{5}}$ is
$$0 \to S^{2}(-9) \to S^{7}(-8) \oplus S^{6}(-7) \to S^{70}(-6) \to S^{154}(-5) \to S^{168}(-4) \to $$ $$S^{98}(-3) \to S^{26}(-2) \to S \to S/I(S_{D_{2\cdot 5}})) \to 0.$$
\end{itemize}
\end{example}

Our next aim is to describe a minimal set of generators of $I(S_{D_{2d}})$. We define a new set of indeterminates $z_{(r,\gamma)}$, we set $S' = k[z_{(r,\gamma)}]$ and we consider the linear change of variables induced by $\rho$ (see \ref{Equation: change of variable}): 
\begin{equation}
\left\{\begin{array}{lllll}
z_{(r,\gamma)} & = & w_{(r,\gamma)}, & \text{if} & w_{(r,\gamma)} \neq w_{(2,0)}\\
z_{(2,0)} & = & w_{(2,0)} + 2w_{(d,d)}, & &  
\end{array}\right.
\end{equation}
which gives an isomorphism $\tilde{\rho}: k[z_{(r,\gamma)}] \to S$ of polynomial rings. We have the following commutative diagram
\begin{center}
\begin{tikzpicture}
\matrix (M) [matrix of math nodes, row sep=3em, column sep=3em]{
S' & k[\mathcal{A}_{2d}] \\
S               & k[\mathcal{B}_{2d}] \\
};
\draw[->] (M-1-1) -- (M-1-2) node[midway,above]{$\psi_{d}$};
\draw[->] (M-2-1) -- (M-2-2) node[midway,above]{$\varphi_{d}$};
\draw[->] (M-1-1) -- (M-2-1) node[midway,left]{$\tilde{\rho}$};
\draw[->] (M-1-2) -- (M-2-2) node[midway,left]{$\rho$};
\end{tikzpicture}
\end{center}
where $\psi_{d}(z_{(r,\gamma)}) = \rho^{-1}(\varphi_{d}(w_{(r,\gamma)}))$  if $z_{(r,\gamma)} \neq z_{(2,0)}$ (see (\ref{Equation: change of variable})) and $\psi_{d}(z_{(2,0)}) = y_2^2$. In particular, $\psi_{d}$ sends bijectively the variables $z_{(r,\gamma)}$ to the monomials of $\mathcal{A}_{2d} = \{y_0^{b_0}y_1^{b_1}y_2^{b_2} \,\mid \, b_0+2b_1 + db_2 = 2d\}$ by the formula $\psi_{d}(z_{(r,\gamma)}) = y_0^{d(2-r) + (d-2)\gamma}y_1^{\gamma}y_{2}^{r-\gamma}$. We obtain the following result. 

\begin{theorem} \begin{itemize} \item[(i)] $\ker(\psi_{d})$ is a binomial ideal of $S'$ minimally generated by quadrics.
\item[(ii)] $I(S_{D_{2d}}) = \tilde{\rho}(\ker(\psi_{d}))$ and a minimal set of generators of $I(S_{D_{2d}})$ is the following sets of binomials and trinomials: 

\vspace{0.3cm}
\noindent\hspace{-0.5cm} $\begin{array}{l}
\{ w_{(r_{1},\gamma_{1})}w_{(r_{2},\gamma_{2})} - w_{(r_{3},\gamma_{3})}w_{(r_{4},\gamma_{4})} \,\mid \, (r_{i},\gamma_{i}) \neq (2,0), \; r_{1} + r_{2} = r_{3} + r_{4}, \; \gamma_{1} + \gamma_{2} = \gamma_{3} + \gamma_{4}\}\\[0.3cm]
\{(w_{(2,0)} + 2w_{(d,d)})w_{(\gamma_1,\gamma_1)} - w_{(r_{2},\gamma_2)}w_{(r_{3},\gamma_{3})} \,\mid \, (r_i,\gamma_i) \neq (2,0), \; \gamma_1 + 2 = r_2 + r_3, \; \gamma_1 = \gamma_2 + \gamma_3\}.
\end{array}$

\end{itemize}  
\end{theorem}
\begin{proof} (i) $\ker(\psi_{d})$ is generated by the set of binomials of the form $\prod_{i=1}^{l} z_{(r_{j_i},\gamma_{j_i})} - \prod_{i=1}^{l} z_{(r_{m_{i}},\gamma_{m_{i}})}$, $l \geq 2$, such that $\prod_{i=1}^{l} \psi_{d}(z_{(r_{j_i},\gamma_{j_i})}) =  \prod_{i=1}^{l} \psi_{d}(z_{(r_{m_{i}},\gamma_{m_{i}})})$ (see \cite[Theorem 1]{Hochster}). From this and Corollary \ref{Corollary: quadrics}, it follows that $\ker(\psi_{d})$ is minimally generated by binomials of degree $2$. Precisely, since we have $\psi_{d}(z_{(r,\gamma)}) = y_0^{d(2-r)+(d-2)\gamma}y_1^{\gamma}y_2^{r-\gamma}$, these binomials are: 
\begin{equation}\label{Equation:binomials}
\{z_{(r_1,\gamma_1)}z_{(r_2,\gamma_2)} - z_{(r_3,\gamma_3)}z_{(r_4,\gamma_4)} \,\mid \, r_1 + r_2 = r_3 + r_4, \, \gamma_1 + \gamma_2 = \gamma_3 + \gamma_4\}.
\end{equation}

(ii) Since $\tilde{\rho}$ and $\rho$ are isomorphisms of $k$-algebras, from the above commutative diagram we have that $I(S_{D_{2d}}) = \tilde{\rho}(\ker(\psi_{d}))$. Applying $\tilde{\rho}$ to (\ref{Equation:binomials}), we obtain the description of the minimal set of generators in (ii). 
\end{proof}
We end this note with an example. 
\begin{example} \rm Fix $d = 4$. We compute the homogeneous ideal $I(S_{D_{2\cdot 4}})$ using the software Macaulay2. It is minimally generated by the $15$ binomials and $5$ trinomials of degree $2$ that we list below.
$$\begin{array}{lcllllcl}
w_{(0,0)}w_{(2,2)} & - & w_{(1,1)}^{2} & \quad \quad \quad & w_{(1,0)}w_{(3,3)} & - & w_{(1,1)}w_{(3,2)}\\
w_{(0,0)}w_{(3,3)} & - & w_{(1,1)}w_{(2,2)} & \quad \quad \quad & w_{(1,0)}w_{(3,3)} & - & w_{(2,1)}w_{(2,2)}\\
w_{(0,0)}w_{(3,2)} & - & w_{(1,0)}w_{(2,2)} & \quad \quad \quad & w_{(1,0)}w_{(4,4)} & - & w_{(2,1)}w_{(3,3)}\\
w_{(0,0)}w_{(2,1)} & - & w_{(1,0)}w_{(1,1)} & \quad \quad \quad & w_{(1,0)}w_{(4,4)} & - & w_{(2,2)}w_{(3,2)}\\
w_{(0,0)}w_{(4,4)} & - & w_{(1,1)}w_{(3,3)} & \quad \quad \quad & w_{(1,1)}w_{(4,4)} & - & w_{(2,2)}w_{(3,3)}\\
w_{(0,0)}w_{(4,4)} & - & w_{(2,2)}^{2} & \quad \quad \quad & w_{(2,1)}w_{(4,4)} & - & w_{(3,2)}w_{(3,3)}\\
w_{(1,0)}w_{(2,2)} & - & w_{(1,1)}w_{(2,1)} & \quad \quad \quad & w_{(2,2)}w_{(4,4)} & - & w_{(3,3)}^2\\
w_{(1,0)}w_{(3,2)} & - & w_{(2,1)}^{2} & \quad \quad \quad & \\
\end{array}$$ 

$$\begin{array}{lclcl}
w_{(1,0)}^{2} & - & w_{(0,0)}w_{(2,0)} & - & 2w_{(0,0)}w_{(4,4)}\\
w_{(1,0)}w_{(2,1)} & - & w_{(1,1)}w_{(2,0)} & - & 2w_{(1,1)}w_{(4,4)}\\
w_{(1,0)}w_{(3,2)} & - & w_{(2,0)}w_{(2,2)} & - & 2w_{(2,2)}w_{(4,4)}\\
w_{(2,1)}w_{(3,2)} & - & w_{(2,0)}w_{(3,3)} & - & 2w_{(3,3)}w_{(4,4)}\\
w_{(3,2)}^{2}      & - & w_{(2,0)}w_{(4,4)} & - & 2w_{(4,4)}^{2}\\
\end{array}$$
\end{example}

\end{document}